\documentclass[12pt]{amsart}
\usepackage[T1]{fontenc}
\usepackage[english]{babel}
\usepackage{xspace}

\usepackage{amsmath,amssymb,amsthm,enumerate,xspace,dsfont}
\usepackage{mathtools}
\usepackage{accents}
\usepackage[colorlinks=true,citecolor=blue,linkcolor=blue]{hyperref}
\usepackage{mathrsfs}
\usepackage[alphabetic]{amsrefs}

\usepackage[margin=1.4in]{geometry}

\numberwithin{equation}{section}

\newtheorem{thm}{Theorem}
\newtheorem*{thm*}{Theorem}

\newtheorem{question}{Question}

\newtheorem*{prop*}{Proposition}
\newtheorem{lem}[thm]{Lemma}
\newtheorem{cor}[thm]{Corollary}
\newtheorem*{cor*}{Corollary}

\newtheorem*{iproblem*}{Problem}

\theoremstyle{definition}

\newtheorem{defi}[thm]{Definition}
\newtheorem*{defi*}{Definition}

\newtheorem*{rem*}{Remark}
\newtheorem*{warn*}{Warning}
\newtheorem*{com*}{Comment}

\newcommand{\bR}{\mathbb{R}}
\newcommand{\bN}{\mathbb{N}}

\newcommand{\cF}{\mathcal{F}}
\newcommand{\cI}{\mathcal{I}}
\newcommand{\cJ}{\mathcal{J}}
\newcommand{\cO}{\mathcal{O}}

\newcommand{\om}{\omega}
\newcommand{\ep}{\varepsilon}

\newcommand{\act}{\curvearrowright}

\title[]{On fixed point property for $L_p$-representations of Kazhdan groups}

\author[]{Alan Czuron}
\address{Alan Czuron\\ University of Warsaw, Poland}
\email{alanczuron@gmail.com}

\author[]{Mehrdad Kalantar}
\address{Mehrdad Kalantar, University of Houston, USA}
\email{mkalantar@uh.edu}

\thanks{A.C. was partially supported by the National Science Centre, Poland, grant no. 2015/18/E/ST1/0021. M.K. was partially supported by the NSF Grant DMS-1700259.}
\begin{document}

\begin{abstract}
Let $G$ be a topological group with finite Kazhdan set, let $\Omega$ be a standard Borel space and $\mu$ a finite measure on $\Omega$. We prove that for any $p\in [1, \infty)$, any affine isometric action $G \act L_p(\Omega, \mu)$ whose linear part arises from an ergodic measure-preserving action $G \act (\Omega, \mu)$, has a fixed point.
\end{abstract}

\maketitle

\thispagestyle{empty}

\section{Introduction}

Kazhdan's property (T) is a fundamental concept in structure theory of groups with many significant applications in various areas of group theory such as ergodic theory, operator algebras, representation theory, etc. 

This notion often appears as a tool in proving rigidity properties of groups and their associated structures.

A closely related property is Serre's property $(\cF H)$, the fixed point property for affine isometric actions on  Hilbert spaces.
In fact these properties are equivalent for all locally compact $\sigma$-compact groups, and Property $(T)$ implies $(\cF H)$ for every topological group \cite{Del, Gui}.

Both properties above are defined for Hilbert space representations, and both have generalizations to similar properties for any given Banach space. The systematic study of these properties and their applications was initiated in the work of Bader-Furman-Gelander-Monod \cite{BFGM}.
 
\bigskip

Let $G$ be a topological group. By an isometric linear action $\pi : G \act X$ of $G$ on a Banach space $X$ we mean a strongly continuous group homomorphism $\pi:G \to \cO(X)$, where $\cO(X)$ is the group of all invertible isometric linear maps on $X$. We say that $\pi$ almost have invariant vectors if $\,\inf_{\|x\|=1} \text{diam} (\pi(K) x ) = 0\,$,
for all compact subsets $K\subset G$.
Denote by $X^{\pi}$ the subspace of all $\pi(G)$-invariant vectors in $X$.

The group $G$ is said to have Property $T_X$ (\cite[Definition 1.1]{BFGM}) if for any isometric linear action $\pi : G \act X$, the quotient action $\tilde{\pi} : G \act X/X^{\pi} $ does not almost have $G$-invariant vectors.  
Then $G$ has Kazhdan's Property $ (T)$ iff it has $T_H$ for every Hilbert space $H$.

The group $G$ is said to have Property $\cF_X$ (\cite[Definition 1.2]{BFGM}) if any isometric affine action $\rho : G \act X$ has fixed points. This is equivalent to vanishing of the first cohomology $H^1(G, \pi)$ of $G$ with $\pi$-coefficients, where $\pi$ is the linear part of $\rho$. Recall that $H^1(G, \pi)$ is defined to be the quotient group $Z^1(\pi)/B^1(\pi)$, where $Z^1(G, \pi)$ is the group of all $\pi$-cocycles and $B^1(\pi)$ is the subgroup of $Z^1(G, \pi)$ consisting of all $\pi$-coboundaries.

\bigskip

It was proved in \cite[Theorem 1.3]{BFGM} that for any locally compact second countable (l.c.s.c.) group $G$, property $\cF_X$ implies $T_X$ for any Banach space $X$. The converse is known to not be true in general. 

We refer the reader to \cite{BFGM}, and references therein, for more details on all these concepts.

\bigskip

One of the most important cases, both in theory and applications, which is also the main interest of this paper, is the above properties for $X$ being the $L_p$-space of a $\sigma$-finite measure space $(\Omega,\mu)$.

By \cite[Theorem A. and Theorem 1.3]{BFGM} any l.c.s.c. group $G$ with property (T) has $T_{L_p(\Omega, \mu)}$ for any $1\leq p <\infty$ and has $\cF_{L_p(\Omega, \mu)}$ for $1< p \le 2$ where $\mu$ is a $\sigma$-finite measure on a standard Borel space $\Omega$.
In fact $(T)$ implies $(\cF_{ L_p})$ with $1\le p\le 2$ for general topological groups \cite{MarSal}.
On the other hand, there are many property (T) groups (even discrete) that do no not have $\cF_{\ell_p}$ for large $p$ (\cite{BoPa, Yu, Nica1}).

Thus, the fixed point properties $\cF_{L_p}$ are in general strictly stronger than property (T), and groups admitting $\cF_{L_p}$ often posses strong rigidity properties. Higher rank semi-simple Lie groups have $\cF_{L_p}$ for all $p\in (1, \infty)$ (\cite[Theorem B]{BFGM}), while rank one groups in general fail to have $\cF_{L_p}$ for large $p$ (see e.g. \cite{Pa}, \cite{CTV}).

But failing to have $\cF_{L_p}$ only implies the existence of certain linear isometric actions $\pi$ of $G$ on an $L_p$-space with non-vanishing first cohomology.
But then this would naturally lead to interesting and important questions of whether there are subclasses of linear isometric $G$-actions on $L_p$-spaces whose first cohomology vanishes, or more generally, can one compute the cohomology of a given such representation?

We recall that by Banach--Lamperti's theorem \cite{Banach, Lamperti}, for $1\leq p<\infty$, $p\neq 2$, every linear isometric action $\pi$ of a group $G$ on $L_p(\Omega,\mu)$ induces a measure-class preserving action of $G$ on the measure space $(\Omega,\mu)$ (which we refer to as the corresponding BL action) such that
\[
\pi_\gamma (f)(\om) = \chi^{(\gamma)}(\om)f_\gamma(\om) \big[\frac{d\gamma \mu }{d\mu}(\omega)\big]^{\frac{1}{p}} 
\]
for all $f\in L_p(\Omega,\mu)$, where $f_\gamma(\cdot)= f(\gamma^{-1}\,\cdot)$ is the left translation of $f$ by $\gamma$, and $\chi^{(\gamma)}$ is a measurable function on $\Omega$ with $|\chi^{(\gamma)}|\overset{{\rm a.e.}}{=} 1$, for every $\gamma\in G$.

In particular, an isometric action $\pi:G \act L_p(\Omega,\mu)$, where $1<p<\infty$ and $p\neq 2$, induces isometric actions $\pi^q:G \act L_q(\Omega,\mu)$ for every $q\in [1,\infty)$ in the natural way. (For the connection between property $\cF_{L_p}$ for various $p\in [1, \infty)$, see e.g. \cite{LO, Alan, MarSal}).

From ergodic theoretical point of view, probability measure preserving (p.m.p.) ergodic actions of a group $G$ are of particular interest and importance, and their properties often reveal significant information about the structure of the group $G$ itself. 
In fact, property (T) and all $T_{L_p}$, $p\in [1, \infty)$, are determined only by the Koopman representations of ergodic p.m.p. actions on the standard Lebesgue space.

Thus, the following question is natural. Say $G$ has $\cF_{L_p}^0$ for $p\in [1, \infty)$, $p\neq 2$, if $H^1(G, \pi) = 0$ for any linear isometric action $G \act L_p(\Omega, \mu)$ whose BL action $G \act (\Omega,\mu)$ is ergodic and measure preserving, where $\Omega$ is a standard Borel space, and $\mu$ is a finite measure on $\Omega$. For $p=2$, $\cF_{L_2}^0$ is defined similarly for unitary representations arising from such actions $G \act (\Omega,\mu)$.

\begin{question}\label{question}{\cite[Question 29]{LO}}
Does every topological group $G$ with property $(T)$ have $\cF_{L_p}^0$ for every $p\in [1, \infty)$?
\end{question}

In \cite[Theorem 3]{LO} authors answer this question in the affirmative in the case of countable discrete groups. Our main result generalizes that to a large class of Kazhdan groups.

\begin{thm}\label{Main}
Let $G$ be a topological group with a finite Kazhdan set. Then $G$ has $\cF_{L_p}^0$ for every $p\in [1, \infty)$.
\end{thm}

Of course all discrete property $(T)$ groups have finite Kazhdan sets. But, the class of non-discrete property (T) groups with finite Kazhdan sets is also quite large. This property was studied in \cite{Shalom1, Bekka, Pes}.
In \cite[Theorem 7]{Bekka} Bekka proved that a locally compact property (T) group $G$ admits a finite Kazhdan set iff its Bohr compactification $bG$ admits a finite Kazhdan set. In particular, every minimally almost periodic group with property (T) has a finite Kazhdan set. This includes all semisimple Lie groups with no compact factors. Thus, we cover the important case of rank one such Lie groups. For instance, the simple Lie groups ${\rm Sp}_{n, 1}(\bR)$ have property (T) but fail to have $\cF_{L_p}$ for $p>4n+2$ \cite{Pa, CTV}. Our Theorem \ref{Main} yields:

\begin{cor}\label{cor:spn}
Let $G={\rm Sp}_{n, 1}(\bR)$, let $(\Omega,\mu)$ be a standard Borel probability space, and let $p\in [1, \infty)$, $p\neq 2$. Suppose that $G \act L_p(\Omega,\mu)$ is a linear isometric action whose BL action $G \act (\Omega,\mu)$ is ergodic and measure preserving. Then $H^1(G, \pi^q) = 0$ for all $q\in [1, \infty)$.
\end{cor}

\bigskip

The proof of Theorem \ref{Main} is considerably more technical compared to the discrete case, and requires some new methods. Specifically, a key new tool in our arguments is the use of weak-$L_p$ estimates, which is to some extent inspired by the proof of Marcinkiewicz interpolation theorem.

In fact, we prove the following result on vanishing of the first cohomology in the weak $L_p$ sense for  property (T) groups (without requiring existence of finite Kazhdan sets).

\begin{thm}\label{thm:wLp}
Let $G$ be a topological group with property (T), let $\Omega$ be a standard Borel space, $\mu$ a finite measure on $\Omega$, and let $p\in [1,\infty)$, $p\neq 2$. Suppose that $\pi: G\act L_p(\Omega,\mu)$ is a linear isometric action such that its BL action $G \act (\Omega,\mu)$ is ergodic and measure preserving. Then for any $\pi$-cocycle $b : G\rightarrow L_p(\Omega,\mu)$ there is $f\in L_{p,w}(\Omega,\mu)$ such that $ b_\gamma =f_\gamma - f$, for all $\gamma\in G$.
\end{thm}

The above theorem is interesting on its own, suggesting a notion of ``weak fixed point property'' which is some step towards weak $L_p$ cohomology theory suggested in \cite{BoSa}.

We remark that the exclusion of the case $p=2$ in both Corollary \ref{cor:spn}  and Theorem \ref{thm:wLp} has no significance (any property (T) group has $\cF_{L_2}$), and the only reason for adding it is that not every $\pi: G\act L_2(\Omega,\mu)$ arises from an action $G\act (\Omega,\mu)$. In fact, since property (T) implies $\cF_{L_p}$ for every $p\in [1,2]$ (see \cite[Theorem A. and Theorem 1.3]{BFGM}, \cite[Theorem 2]{LO}, and \cite[Corollary 2]{MarSal}), Question \ref{question} only remains for $p>2$.


\section{Fixed points in weak-$L_p$ spaces}
In this section we prove Theorem \ref{thm:wLp}.
But first, let us quickly review the basic relevant definitions and facts about these spaces. 

\subsection{Weak $L_p$ spaces}
Let $f:(\Omega,\mu) \rightarrow \mathbb{R}$ be a measurable function.
The distribution function $\lambda_f:\mathbb{R}_{+} \rightarrow [0 \,,\, \mu(\Omega)]$ of $f$ is defined by
\[
\lambda_f(x) = \mu (\left\{\, \omega \in \Omega : |f(\om)|>x \,\right\})\,.
\]
\begin{defi}
The \emph{weak-$L_p$ norm} of the measurable function $f$ is defined by
\[
\|\,f\,\|_{p,w}\,=\, \sup_{x>0} \,[x \,\lambda_f(x)^{\frac{1}{p}}]\,.
\]
We define the \emph{weak-$L_p$ space} of $(\Omega,\mu)$ to be 
\[
L_{p,w}(\Omega,\mu) \,:=\, \{f \,:\,  \|\,f\,\|_{p,w}<\infty \} \,.
\]
\end{defi}
Note that $\|\,\cdot\,\|_{p,w}$ is not truly a norm as it does not satisfy the triangle inequality. 
Obviously, $\|\,f\,\|_{p,w}\leq \|\,f\,\|_p$ for any $f\in L_p(\Omega,\mu)$, and therefore $L_p(\Omega,\mu) \subset L_{p,w}(\Omega,\mu)$. Moreover it is known that for every $\ep>0$ we have $L_{p,w}(\Omega,\mu) \subset L_{p-\ep} (\Omega,\mu)$.

Moreover, for $f\in L_1(\Omega,\mu)$ and $p\in [1,\infty)$, using Fubini's theorem we have
\begin{equation}\label{FirstLemma}
\begin{split}
p\int_{0}^{\infty} x^{p-1}\lambda_f(x)\,dx &=
p\int_{0}^{\infty}x^{p-1} \mu( \om\in \Omega : |f(\om)|>x  ) \,dx \\&=
p\int_{0}^{\infty} x^{p-1} \int_{\Omega} \chi_{|f|>x} (\om)\,d\mu(\om)\,dx \\&=
\int_{\Omega} \int_{0}^{|f(\om)|} px^{p-1}\,dx\, d\mu(\om) \\&=
\int_{\Omega} |f(\om)|^p \,d\mu(\om) \,=\, 
\|\,f\,\|_p^p \,.
\end{split}
\end{equation}

\bigskip

Before getting to the proof of Theorem \ref{thm:wLp}, we record the following simple fact which will be used frequently throughout the paper. Let $S$ be a symmetric Kazhdan set for a property (T) group $G$. If $G\act (\Omega,\mu)$ is an ergodic measure preserving action on a finite measure space $(\Omega,\mu)$, then there is $\ep>0$ such that 
\begin{equation}\label{SpectralGap}
\sup_{\gamma\in S} \mu(\gamma A - A)> \ep \mu(A) ~~~~~ \text{for all measurable } A\subset \Omega \text{ with } \mu(A)<\frac{1}{2} .
\end{equation}
In fact, since $S$ is a Kazhdan set, there is $\epsilon'>0$ such that $\sup_{\gamma\in S} \|f_\gamma - f \|_2^2 > \epsilon'  \| f \|_{2}^2$ for all $f\in L_2(\Omega,\mu)$ with $\int f d\mu =0$. Applying this to the function $f=\mathds{1}_{A} - \mu(A)\mathds{1}_{\Omega}$ yields \eqref{SpectralGap}.

\bigskip 

\noindent
{\it Proof of Theorem \ref{thm:wLp}.} 
Since $G$ has property (T), it has $\cF_{L_p(\Omega, \mu)}$ for every $1\leq p\leq 2$ by \cite{Del, Gui, MarSal}. Thus, we may assume $p>2$. As $L_p(\Omega,\mu)\subset L_2(\Omega,\mu)$, there exists $f\in L_2(\Omega,\mu)$ such that $b_\gamma = f_\gamma - f$ for all $\gamma\in G$. We will show $f\in L_{p,w}(\Omega,\mu)$. For the sake of contradiction, assume otherwise. For $a<b\in \mathbb{R}\cup\{\infty\}$, denote $I_{a}^{b} = f^{-1}\big([a, b)\big)$. Since $\mu$ is a finite measure there exists  $N_0> 0$  such that $\mu(I_{N_0}^{\infty})<\frac{1}{2}$.
Let $S$ be a compact symmetric Kazhdan set for $G$. By \eqref{SpectralGap} there exists $\ep>0$ such that for all $x> N_0$,
\begin{equation}\label{JSEquation} 
\sup_{\gamma\in S} \mu(s^{-1} I_{x}^{\infty} \cap I_{0}^{x}) \ge \ep \mu ( I_{  x }^{\infty}) \,.
\end{equation}
Set $M= \sup_{\gamma\in S} \|f_\gamma - f\|_{L_{p,w}}$. Since $f_\gamma - f\in L_p(\Omega,\mu)$ and $\|f_\gamma - f\|_{{p,w}}\leq \|f_\gamma - f\|_{p}$ for all $\gamma\in G$,
it follows that the map $\gamma\mapsto \|f_\gamma - f\|_{L_{p,w}}$ is continuous on $G$, and in particular $M< \infty$.

For any $x> N_0$ and $y< x$, we have
\begin{equation}\label{calcs2}
	    \sup_{\gamma\in S} \mu\left(s^{-1} I_{x}^{\infty}  \cap I_0^{y} \right)
	    \le 
	    \sup_{\gamma\in S} \lambda_{\pi^2(s) f - f}
	    (x - y ) \le 
	    \frac{M^p}{(x-y)^p} \,.
\end{equation}
Let $\beta = (1+\frac\ep4)^{\frac1p}$. For $x\in \mathbb{R}$ denote $C_x = x\lambda_f (x)^{\frac{1}{p}}$. 
If  $x> N_0$ is such that
	    \begin{equation}\label{ConditionL1}
	    C_x \ge \frac{M\beta}{(\frac\ep4)^{\frac{1}{p}}(\beta - 1)}\,,
	    \end{equation}
then for $\displaystyle y=\frac{x}{\beta}$ in \eqref{calcs2} we get
	    \begin{equation}\label{L1Inequality}
	    \sup_{\gamma\in S} \mu\left(s^{-1}  I_{x}^{\infty} \cap I_0^{y}\right)
	    \le \frac{M^p}{(x- y)^p} = \frac{M^p}{x^p (1 - \frac{1}{\beta})^p} 
	    \le \frac{\ep \,C_x^p}{4x^p} 
	    = \frac\ep4 \,\mu( I_x^{\infty}) \,.
	    \end{equation}
On the other hand, since $f \notin L_{p,w}(\Omega,\mu)$, there exists a sequence of natural numbers $n_k\nearrow \infty$ such that $C_{n_k} \nearrow \infty$. We may assume without loss of generality that $n_1>N_0$. For each $i\in \mathbb{N}$ let $x_i  = N_0 \beta^i$. Observe that if $x_{i_k} \le n_k \le x_{i_{k} +1}$ for $k, i_k \in \mathbb{N}$, then
\[
\frac{C_{x_{i_k}}}{x_{i_k}}  = \lambda_f (x_{i_k})^{\frac{1}{p}}\ge \lambda_f ({n_k})^{\frac{1}{p}}=\frac{C_{n_k}}{n_k} \geq \frac{C_{n_k}}{x_{i_{k}+1}} = \frac{C_{n_k}}{\beta x_{i_k}},
\]
which implies $\displaystyle C_{x_{i_k}} \ge  \frac{C_{n_k}}{\beta}$. So in particular $\limsup_i C_{x_{i}} = \infty$. Thus, for each $m\in \mathbb{N}$, there is $j \in \mathbb{N}$ such that $\displaystyle  C^p_{x_m} < \frac{\beta^p-1}{\beta^p} C^p_{x_{m+j}}$; this implies that there is $i\geq m$ such that $\displaystyle \mu( I_{x_i}^{x_{i+1}}) < \frac\ep4 \mu( I_{x_{i+1}}^{\infty})$, for otherwise we would have 
\begin{equation}\label{calcs3}
\begin{split}
C_{x_m}^p &\geq x_m^p\mu(  I_{x_m}^{x_{m+1}}) 
\ge x_m^p \frac\ep4 \mu(  I_{x_{m+1}}^{\infty} ) 
= x_m^p \frac\ep4 [\mu( I_{x_{m+1}  }^{x_{m+2}  }) + \mu( I_{x_{m+2}  }^{\infty}  )]   
\\&\ge x_m^p\frac\ep4\beta^p  \mu(  I_{x_{m+2}  }^{\infty})
 \ge \cdots 
 \ge x_m^p\frac\ep4\beta^{p(j-1)}\mu( I_{x_{m+j}  }^{\infty})
\\&=\frac\ep4\beta^{p(j-1)} \frac{C_{x_{m+j}}^p}{\beta^{jp}} 
= \frac{\beta^p-1}{\beta^p} C^p_{x_{m+j}} \,.
\end{split}
\end{equation}
So, by the above we can choose $m_0\in\mathbb{N}$ such that $\displaystyle C_{x_{m_0}} \ge  \frac{M\beta^2}{(\frac{\epsilon}{4})^{ \frac{2}{p}} (\beta - 1) } $ and such that $\displaystyle \mu( I_{x_k}^{ x_{k+1}})< \frac\ep4 \mu( I_{x_{k+1}}^{\infty})$ for some $k<m_0$;
let $m_1\in \mathbb{N}$ be the largest such $k<m_0$. 
Then by inequalities \eqref{calcs3} we have 
\[
C_{x_{m_1 +1}} \ge \frac{(\beta^p-1)^{\frac{1}{p}}}{\beta} \,C_{x_{m_0}} \ge \frac{M\beta}{(\frac\ep4)^{\frac{1}{p}}(\beta - 1)}\,.
\]
Thus, $C_{x_{m_1+1}}$ satisfies \eqref{ConditionL1}, and hence for $x= x_{m_1+1}$ and $y = x_{m_1}$ we have \eqref{L1Inequality}, 
which combined with the inequality $\displaystyle \mu( I_{x_{m_1}}^{ x_{m_1+1}})< \frac\ep4 \mu(  I_{x_{m_1+1}}^{\infty})$ from the choice of $m_1$, yield
\[\begin{split}
\sup_{\gamma\in S} \mu(s^{-1}  I_{x_{m_1+1}}^{\infty} \cap  I_{0}^{x_{m_1+1}}) 
&\leq 
\sup_{\gamma\in S}\left[\mu\left(s^{-1}  I_{x_{m_1+1}}^{\infty} \cap I_0^{x_{m_1}}\right)+ \mu\left( I_{x_{m_1}}^{x_{m_1+1}}\right)\right]  
\\&\leq 
\frac\ep4 \,\mu( I_{x_{m_1+1}}^{\infty}) + \frac\ep4 \,\mu( I_{x_{m_1+1}}^{\infty})
\,=\, \frac{\ep}{2} \,\mu( I_{x_{m_1+1}}^{\infty}) \,,
\end{split}\]
which contradicts \eqref{JSEquation}. This concludes the proof. \qed

\section{Proof of Theorem \ref{Main}}

This section is devoted to the proof of Theorem \ref{Main}. The main work towards that is the long proof of the following lemma.

\begin{lem}\label{lem:main}
Let $G$ be a topological group that has a finite Kazhdan set $S$. Let $(\Omega,\mu) = ([0,1], {\rm Lebesgue})$, and let $G\act (\Omega,\mu)$ be an ergodic measure preserving action.
For $p>2$, if $f\in L_{p,w}(\Omega,\mu)$ is non-negative and such that $f_\gamma-f \in L_p(\Omega,\mu)$ for all $\gamma\in G$, then $f\in L_p(\Omega,\mu)$.
\end{lem}
 
\begin{proof}
We assume $S$ is symmetric, and we choose $0<\ep<1$ so that \eqref{SpectralGap} holds. Choose and fix constants $\delta$ and $\alpha$ such that $0<\delta < \frac{\ep}{8}$ and $1<\alpha^p<1+\delta$. 

For $k\in\bN$ let $A_k = \{\omega\in \Omega: |f(\omega)| \ge \alpha^k \}$. Upon rescaling the function $f$ we may assume without loss of generality that $\mu(A_{1} )< \frac{1}{2}$.

Denote by $\mathbb{N}_1$ the set of $k\in \mathbb{N}$ such that
$\,\mu(A_k  -  A_{k+1})\ge \delta \mu(A_{k+1})\,$,
and $\mathbb{N}_2 = \mathbb{N} - \mathbb{N}_1$.

Let $\mathbb{N}_1^0=\{m\in \mathbb{N}_1\,:\, m-1 \notin \mathbb{N}_1\}$, and for $m\in \mathbb{N}_1^0$ let 
$$\Delta(m) = \min \{l\in \mathbb{N}_2 : l > m \} \,;$$ 
so $k\in \mathbb{N}_1$ for all $m\leq k \leq \Delta(m)-1$. 

For $k\in \bN_2$, observe that 
\begin{equation}\label{eq:ineq-N_2}
\sup_{\gamma\in S} \mu(\gamma A_{k+1} -A_{k}) \ge \frac{7\ep}{8}\mu(A_{k+1}) .
\end{equation}
We claim that if the set $\mathbb{N}_2$ is finite then $f \in L_p(\Omega,\mu)$.
To see this, suppose $N\in \mathbb{N}$ is such that $k\in \mathbb{N}_1$ for every $k\geq N$. Then
\[\begin{split}
\mu(A_N) &= \mu(A_N  -  A_{N+1}) + \mu(A_{N+1}) \ge (1+\delta)  \mu(A_{N+1}) \\&=
(1+\delta) [\mu(A_{N+1}  -  A_{N+2})  + \mu(A_{N+2}) ] 
\\&\ge (1+\delta)^2 \mu (A_{N+2}) \ge \cdots 
\ge (1+\delta)^m \mu(A_{N+m})
\end{split}
\]
for all $m\in \mathbb{N}$.
Thus, it follows
\begin{equation}\label{meas-est-N1}
\begin{split}
\|f_{|_{A_N}}\|_{L_p(\Omega,\mu)}^p &\le \sum_{i=0}^{\infty} \alpha^{p(N+1 + i)} \mu (A_{N+i}  -  A_{N+i +1} ) \\ &\le
\alpha^{p(N+1)} \sum_{i=0}^{\infty} \alpha^{pi} \mu(A_{N+i}) \\&\le 
\alpha^{p(N+1) } \sum_{i=0}^{\infty} \frac{\mu(A_N)}{~(1+\delta)^i~}\, \alpha^{pi} \,<\, \infty \,.
\end{split}
\end{equation}
Since $f_{|_{\Omega   -  A_N}}$ is bounded, the claim follows. Thus, we assume in the following that the set $\mathbb{N}_2$ is infinite.

We divide the set $\Omega$ into the union of two families of disjoint measurable sets 
$$\mathcal{I} : = \bigcup_{m\in \mathbb{N}_1^0}\,(A_{m} - A_{\Delta(m)-1 })
~~~~~~~~~
and
~~~~~~~~~
\mathcal{J} : = \bigcup_{n\in \mathbb{N}_2}\, (A_{n} - A_{n + 1})\,.$$

The proof will be divided into two parts, proving both restrictions of $f$ to $\cI$ and $\cJ$ are in $L_p(\Omega,\mu)$, and that obviously yields the result.

\bigskip

\noindent
{\bf The $\cI$ part:}\ We first prove that the restrction $f_{|_{\cI}} \in L_p(\Omega,\mu)$.
If the set $\mathbb{N}_1^0$ is finite, then $f_{|_{\cI}}$ is bounded since the set $\mathbb{N}_2$ is infinite, and hence there is nothing to prove. So we also assume that $\mathbb{N}_1^0$ is an infinite set and that $m_1 < m_2 < ...$ is an enumeration of $\mathbb{N}_1^0$. 

For $i\in\bN$, noting that $m_i - 1 \not \in \mathbb{N}_1$, we get 
$$\mu(A_{m_i - 1} - A_{m_i}) < \delta \mu(A_{m_i}) <\frac{\ep}{8} \mu(A_{m_i}) \,.$$
On the other hand we know from Lemma \ref{SpectralGap} that 
$$\sup_{\gamma\in S} \mu (\gamma A_{m_i} - A_{m_i}) > \ep \mu(A_{m_i}) \,.$$
The above two inequalities imply
\begin{equation}
\begin{split}
\sup_{\gamma\in S} \mu (\gamma A_{m_i} - A_{m_i - 1}) > \frac{7}{8}\ep \mu(A_{m_i}) \ge \frac{7}{8}\ep \,\sum_{k=i}^{\infty} \mu(A_{m_k} - A_{\Delta(m_{k})})   .
\end{split}
\end{equation}
For each $i\in \bN$ choose $\gamma_i \in S$ that attains the supremum in the above, and let
\[
B_i = \gamma_iA_{m_i} - A_{m_i - 1} .
\]
Now, let $D_0 = \emptyset$ and choose inductively, for each $i\in \bN$ the set $D_i\subset B_i - \cup_{j=1}^{i-1} D_j$ such that
\[
\mu(D_i) > \frac{5\ep}{8} \,\mu(A_{m_i} - A_{\Delta(m_{i})})
\] 
and 
\[
\mu(B_{i+1} - \cup_{j=1}^{i} D_j) \geq  \frac{7\ep}{8} \,\sum_{k=i+1}^{\infty} \mu(A_{m_k} - A_{\Delta(m_{k})}) \,.
\]
Then $D_i$'s are pairwise disjoint and for all $\om\in D_i$,
\[
\sup_{\gamma\in S} \left|\, f_\gamma (\omega) - f(\omega) \,\right| >  \alpha^{m_i} - \alpha^{m_{i}-1} = (\frac{\alpha - 1}{\alpha})\, \alpha^{m_{i}} \,.
\]
Let $\displaystyle c_1 = \sum_{j=1}^{\infty}  \frac{\alpha^{pj}}{~(1+\delta)^{j-1}~}$. Then, calculations similar to \eqref{meas-est-N1} give
\[
\left\|\, f_{|_{A_{m}  - A_{\Delta(m)}}} \,\right\|_p^p 
~\le \,
c_1  \mu(A_{m})\alpha^{pm}
\]
for all $m\in \bN_1$, and therefore
\[\begin{split}
\left\|\, f_{|_{A_{m_i}  - A_{\Delta(m_i)}}} \,\right\|_p^p 
&\le \,
\mu(A_{m_i}-A_{m_i+1})\alpha^{p(m_i+1)} + c_1 \mu(A_{m_i+1})\alpha^{p(m_i+1)} 
\\&\le \frac{\alpha^{p}(1+c_1)}{\delta}\,\mu(A_{m_i}-A_{m_{i+1}})\alpha^{pm_i}
\\&\le c_2\,\mu(A_{m_i}-A_{\Delta(m_i)}) \,\alpha^{pm_i}
\\&\le \,
\frac{c_2\alpha^{p}}{(\alpha-1)^p } \,
\mu\left(A_{m_i}  - A_{\Delta(m_i)}\right)\, \sup_{\gamma\in S} \left|\, f_\gamma (\omega) - f(\omega) \,\right| ^p
\end{split}\]
for all $i\in \bN$ and $\om\in D_i$, where $c_2 = \frac{\alpha^{p}(1+c_1)}{\delta}$. 
This implies
\[\begin{split}
\left\|\, f_{|_{A_{m_i}  - A_{\Delta(m_i)}}} \,\right\|_p^p \mu(D_i)
&\le \,
\frac{c_2\alpha^{p}}{(\alpha-1)^p } \,\mu\left(A_{m_i}  - A_{\Delta(m_i)}\right)\, \sum_{\gamma\in S}\left\|\,f_\gamma - f \,\right\|_p^p
\\&\le \,
c_3 \,
\mu\left(D_i\right)\, \sum_{\gamma\in S}\left\|\,f_\gamma - f \,\right\|_p^p ,
\end{split}\]
where $c_3 = \frac{8c_2\alpha^{p}}{(\alpha-1)^p 5\ep}$. 
Hence, we get
\[\begin{split}
\left\|\, f_{|_\cI} \,\right\|_p^p 
~&= \,
\sum_{i=1}^{\infty} \,
\left\|\, f_{|_{A_{m_i}  - A_{\Delta(m_i)}}} \,\right\|_p^p 
~\le \,
c_3\, \, \sum_{\gamma\in S}\, \left\|\,f_\gamma - f \,\right\|_p^p 
\,< \infty \,.
\end{split}\]
This concludes the proof of the $\cI$ part. 

\bigskip

\noindent
{\bf The $\cJ$ part:}\ Next, we prove $f_{|_{\cJ}} \in L_p(\Omega,\mu)$. 
Note from the definition that $\sum_{t=0}^{\infty } \alpha^{ m_t }  \mathds{1}_{ A_{m_t } - A_{ \Delta ( m_t ) - 1 } } \le f_{|_{\cI}}$, and since $f_{|_{\cI}}\in L_p$, it follows $f\in L_p$ iff $f_{|_{\cJ}}+ \sum_{t=0}^{\infty } \alpha^{ m_t }  \mathds{1}_{ A_{m_t } - A_{ \Delta ( m_t ) - 1 } } \in L_p$. So without loss of generality, we may, and will, assume $\{\om: \alpha^m <f(\om)<\alpha^{m+1}\} = \emptyset$ for all $m\in \bN_1$. 

For $z\in \mathbb{R}_+$, define the functions
\[
 \psi^{(z)}(\om) = \begin{cases}
	f(\omega) + \int_{f\le z} f  & \text{ if } f(\omega) > z  \\ 
	\int_{f  \le z} f & \text{   otherwise } , 
	\end{cases}
\]
	and
\[
\zeta^{(z)} (\omega) = 
\begin{cases}
f(\omega) - \int_{f\le z} f & \text{ if } f \le z  \\
-\int_{f \le z} f  & \text{   otherwise } . 
\end{cases}
\]
Observe that for all $z\in \bR_+$, $f = \psi^{(z)} + \zeta^{(z)}$. Our strategy is to handle $f_\gamma-f$ by analyzing the functions $\psi_\gamma^{(z)} - \psi^{(z)}$ and $\zeta_\gamma^{(z)}  - \zeta^{(z)}$ separately.

By definitions, $\{\om : \psi^{(z)}(\om) > z \}= \{\om : f(\om) > z \}$. Moreover, since $f\in L_1(\Omega, \mu)$, it follows $\mu(\{\om : \psi^{(z)}(\om) > z \}) = \mu(\{\om : \psi^{(z)}(\om) \ge z \})$ for almost every $z\in \bR_+$. Recall that we have rescaled $f$ so that $\mu(\{\om : f(\om) > \alpha \}) < \frac12$. Hence, using \eqref{SpectralGap} we get
\begin{equation}\label{ineq:psi}
\begin{split}
\sup_\gamma \|\psi_\gamma^{(z)} - \psi^{(z)}\|_1 
&\geq
\sup_\gamma \int_{\{\om : f(\om) > z ,\, f_\gamma(\om) \leq z\}} f(\omega)\,d\om
\\&\geq
z  \sup_\gamma \mu (\{\om : f(\om) > z ,\, f_\gamma(\om) \leq z\} ) 
\\&\geq
z\ep\,\mu( \{\om : f(\om) > z \} ) = z\ep \lambda_{\psi^{(z)}} (z)
\end{split}
\end{equation}
for a.e. $z\in [\alpha, \infty)$. Also, since $\zeta^{(z)}$ has zero average, by property (T) we have
\begin{equation}\label{ineq:zeta}
\lambda_{\zeta^{(z)}}(z)\le z^{-2p} \|\zeta^{(z)}\|_{2p}^{2p} \le z^{-2p} c_1 \sup_{\gamma\in S} \|\zeta_\gamma^{(z)}  - \zeta^{(z)}\|_{2p}^{2p} 
\end{equation}
for some constant $c_1$ and all $z\in \bR_+$. Using \eqref{FirstLemma}, \eqref{ineq:psi}, \eqref{ineq:zeta}, and the fact that  $\int_{0}^{\alpha} z^{ p-1 } \lambda_{\phi} (z) dz<\infty$ for any measurable function $\phi$, we get 
\begin{equation*}\label{eq:A5}
\begin{split}
\|f\|_p^p &= 
p2^p \int_{0}^{ \infty } z^{ p-1 } \lambda_{ f } (2z)\, dz 
\le 
p2^p \int_{0}^{ \infty } z^{ p-1 } \left(\lambda_{\psi^{(z)}} (z) + \lambda_{\zeta^{(z)}} (z)\right) dz
\\ &\le  
c\, \sum_{\gamma\in S}\, \big(\,{\underbrace{\int_\alpha^{\infty} z^{p-2} \|\psi_\gamma^{(z)} - \psi^{(z)} \|_1 \,dz}_{\textcircled{{\tiny I}}}} \,+ 
\,{\underbrace{\int_\alpha^{\infty} z^{-p-1} \|\zeta_\gamma^{(z)} - \zeta^{(z)}\|_{2p}^{2p} \,dz}_{\textcircled{{\tiny II}}}} \,\big) 
 \end{split}
\end{equation*}
for some constant $c$. In the remaining we prove the sum over $\gamma\in S$ of each of the above two integrals is finite.


\subsection{Estimate I}
In this section we prove that the sum over $\gamma\in S$ of the integral \textcircled{{\tiny I}} is finite. For this, we prove that there is a constant $c>0$ such that 
\begin{equation}\label{eq:A6}
\sum_{\gamma\in S} \| \psi_\gamma^{(z)} - \psi^{(z)} \|_1
\le 
c\, \sum_{\gamma\in S}\int_{|f_\gamma - f| \ge az} |(f_\gamma - f)(\omega)|  d\omega 
\end{equation}
for almost every $z\in [\alpha, \infty)$, where $a = \min\{\frac{\ep}{24 |S|}, \frac{\alpha-1}{\alpha}\}$. Then, using \eqref{eq:A6} and Fubini's Theorem we conclude
\[
\begin{split}
\sum_{\gamma \in S} \int_{\alpha}^{\infty} z^{p-2} \|\psi_{\gamma}^{(z)} - \psi^{(z)} \|_1 dz
&\le
c\,\sum_{\gamma \in S} \int_0^{\infty}  z^{p-2}\int_{|f_\gamma - f| \ge az} |(f_\gamma - f)(\omega)| \,d\om \,dz
\\ &= 
c\,\sum_{\gamma \in S} \int_{\Omega} |(f_\gamma - f)(\om)| \int_0^{\frac{1}{a} |(f_\gamma - f)(\om)|} z^{p-2}\, dz \, d\om
\\ &= 
c\, \sum_{\gamma \in S}
\frac{1}{(p-1)a^{p-1}} \int_{\Omega}| (f_\gamma - f  )(\om) |^p \,d\om < \infty .
\end{split}
\]
Now, we proceed to prove \eqref{eq:A6}. For each $z\in [\alpha, \infty)$ and $\gamma\in S$, define the sets 
\[
P_{\gamma,z} : = \{\om \in \Omega : |f_\gamma (\om)- f (\om) | = |\psi_\gamma^{(z)} (\om) - \psi^{(z)} (\om) | \not= 0\},
\]
\[
Q_{\gamma,z} : = \{\omega \in \Omega - P_{\gamma,z} :  |\psi_\gamma^{(z)} (\omega) - \psi^{(z)} (\omega) | \not= 0\} .
\]

\noindent
\pmb{Case 1: $\displaystyle\sum_{\gamma\in S} \int_{P_{\gamma,z} } |\psi_\gamma^{(z)} - \psi^{(z)}|d\omega \ge 
\sum_{\gamma\in S}  
\int_{Q_{\gamma,z} } |\psi_\gamma^{(z)} - \psi^{(z)}|d\omega$.}\\
For $\gamma\in S$, $k\in \mathbb{N}$, and $z \in [\alpha^k , \alpha^{k+1})$, we have $$\mu(P_{\gamma,z}) \le 3 \mu (A_k) \le 3(1+\delta)\mu (A_{k+1}) \le 6 (A_{k+1})$$ if $k\in \mathbb{N}_2$, and we have $$\mu(P_{\gamma,z})  \le 3 \mu(\{\omega : f (\om) > z \}) = 3\mu(A_{k+1})$$ if $k\in \mathbb{N}_1$, since in this case $(\alpha^k, \alpha^{k+1})\cap {\rm Range}(f)= \emptyset$.
Thus, 
\[
\begin{split}
&\sum_{\gamma\in S} \int_{ \{\omega : |f_\gamma (\om) - f (\om)|\ge az\}} |f_\gamma (\om)- f (\om) | d\omega 
\\\ge &\,
\sum_{\gamma\in S} \int_{ P_{\gamma,z} \cap \{\omega : |f_\gamma (\om) - f (\om)|\ge az\}} |f_\gamma (\om) - f (\om) | d\omega
\\ =&\,
\sum_{\gamma\in S} \int_{ P_{\gamma,z}} |f_\gamma (\om) - f (\om) | d\omega
-
\sum_{\gamma\in S} \int_{ P_{\gamma,z} \cap \{\omega : |f_\gamma - f|(\omega)< az\}} |f_\gamma (\om) - f (\om) | d\omega
\\\ge&~
\frac{1}{2} \sum_{\gamma\in S}  \|\psi_\gamma^{(z)} - \psi^{(z)}\|_1 
-
az \sum_{\gamma\in S}  \mu(P_{\gamma,z}) 
\\\ge&~
\frac{1}{2} \sum_{\gamma\in S} \|\psi_\gamma^{(z)} - \psi^{(z)}\|_1 - 6az |S| \mu(A_{k+1})
\\\underset{(*)}{\ge}&~
\frac{1}{2} \sum_{\gamma\in S} \|\psi_\gamma^{(z)} - \psi^{(z)}\|_1 - 6az |S| \frac{1}{z\ep} \sum_{\gamma\in S} \|\psi_\gamma^{(z)} - \psi^{(z)}\|_1
\\\ge&~
\frac{1}{4} \sum_{\gamma\in S} \|\psi_\gamma^{(z)} - \psi^{(z)}\|_1 
\end{split}
\]
for a.e. $z\in [\alpha^k , \alpha^{k+1})$, where we used \eqref{ineq:psi} for inequality $(*)$.

\bigskip

\noindent
\pmb{Case 2: $\displaystyle \sum_{\gamma\in S} \int_{P_{\gamma,z} } |\psi_\gamma^{(z)} - \psi^{(z)}|d\omega < \sum_{\gamma\in S} \int_{Q_{\gamma,z} } |\psi_\gamma^{(z)} - \psi^{(z)}|d\omega$.}\\
Note that from the definition of the functions $\psi^{(z)}$, for $\om \in Q_{\gamma,z}$, $\gamma\in S$ and $z\in [\alpha , \infty)$ we have $|\psi_\gamma^{(z)}(\om) - \psi^{(z)}(\om)| = \max \{ f_\gamma (\omega), f (\omega)\}$. Define the sets
\[Q_{\gamma,z}^+ = \left\{\omega \in Q_{\gamma,z} : |f_\gamma(\om) - f(\om)| \ge (1-\frac{1}{\alpha}) \max \{ f_\gamma (\omega), f (\omega)\}\right\} ,\] 
and $Q_{\gamma,z}^- = Q_{\gamma,z} - Q_{\gamma,z}^+$ for $\gamma\in S$ and $z\in [\alpha , \infty)$. Then
\[
\begin{split}
\sum_{\gamma\in S} \int_{ Q_{\gamma,z}^+ } |\psi_\gamma^{(z)} - \psi^{(z)}| d\omega
&\le
(\frac{\alpha}{\alpha-1}) \sum_{\gamma\in S}  \int_{ Q_{\gamma,z}^+ } |f_\gamma - f| d\omega
\\&\underset{(*)}{\le}
(\frac{\alpha}{\alpha-1}) \sum_{\gamma\in S} \int_{|f_\gamma - f| \ge (1-\frac{1}{\alpha}) z} |f_\gamma - f| d\omega
\\&\le
(\frac{\alpha}{\alpha-1}) \sum_{\gamma\in S} \int_{|f_\gamma - f| \ge a z } |f_\gamma - f| d\omega ,
\end{split}
\]
where for inequality $(*)$ we used the fact that $z\leq \max \{ f_\gamma (\omega), f (\omega)\}$ on $Q_{\gamma,z}$.

On the other hand, it is seen from definitions that if $\alpha^k \le z < \alpha^{k+1}$ for some $k\in\mathbb{N}_1$, then $Q_{\gamma,z}^- = \emptyset$, and there is no more cases left to consider. Let $\alpha^k \le z < \alpha^{k+1}$ for some $k\in\mathbb{N}_2$. If $k-1\in\mathbb{N}_1$, then $Q_{\gamma,z}^- \subseteq A_k-A_{k+2}$, and so $\mu(Q_{\gamma,z}^-)\le (1+\delta) \mu (A_{k+1})$; if $k-1\in\mathbb{N}_2$, then $Q_{\gamma,z}^- \subseteq A_{k-1}-A_{k+2}$, and so $\mu(Q_{\gamma,z}^-)\le (1+\delta)^2 \mu (A_{k+1})$. Thus, in either case, by \eqref{eq:ineq-N_2}, we get for every $\gamma\in S$,
\[
\mu(Q_{\gamma,z}^-)\le \frac{8}{7\ep}(1+\delta)^2 \sup_{\gamma'\in S} \mu(\gamma' A_{k+1} - A_{k}) ,
\]
and therefore
\[
\begin{split}
&\int_{Q_{\gamma,z}^- } |(\psi_{\gamma}^{(z)}(\om) - \psi^{(z)} (\om) | d\om \le 
\frac{8}{7\ep}(1+\delta)^2 \alpha^{k+2} \sup_{\gamma'\in S} \mu(\gamma' A_{k+1} - A_{k})
\\ \le & \,
\left(\frac{8 \alpha^{2} (1+\delta)^2} {7\ep(1-\frac{1}{\alpha})}\right) \,\sup_{\gamma'\in S} \int_{| f_{\gamma'} - f | \ge az}|( f_{\gamma'} - f)(\omega) |  d \omega ,
\end{split}
\]
where the last inequality follows from the facts that $$\gamma' A_{k+1} - A_{k}\subseteq \{\om: | f_{\gamma'}(\om) - f (\om)| \ge \alpha^{k+1}(1-\frac{1}{\alpha})\}$$ for all $\gamma'\in S$, and that $ az < \alpha^{k+1}(1 - \frac{1}{\alpha}) $.

Combining the above, we get
\[
\begin{split}
&\sum_{\gamma\in S} \|\psi_\gamma^{(z)} - \psi^{(z)} \|_1
\le
2 \sum_{\gamma\in S} \int_{Q_{\gamma,z} } |(\psi_{\gamma}^{(z)}(\om) - \psi^{(z)} (\om) |d\omega
\\ =~& ~
2 \sum_{\gamma\in S} \int_{Q_{\gamma,z}^+ } |(\psi_{\gamma}^{(z)}(\om) - \psi^{(z)} (\om) |d\omega
 +  
2 \sum_{\gamma\in S} \int_{Q_{\gamma,z}^- } |(\psi_{\gamma}^{(z)}(\om) - \psi^{(z)} (\om) |d\omega
\\ \le~& ~
c\, \sum_{\gamma\in S} \int_{| f_\gamma - f | \ge az}|f_\gamma(\omega) - f(\omega) |  d \omega
\end{split}
\]
for some constant $c$. This concludes the proof of finiteness of the integral \textcircled{{\tiny I}}.


\subsection{Estimate II}
In this last section we prove that the integral \textcircled{{\tiny II}} is finite.

For $z\in [\alpha, \infty)$ and $\gamma\in S$, let 
\[E_{\gamma, z} := \big\{\omega \in \Omega : |\zeta_\gamma^{(z)}(\om) - \zeta^{(z)}(\om)| \leq \frac{1}{\alpha-1} |f_{\gamma}(\om) - f(\om)| \big\} .\]
Then
\[\begin{split}
&\int_{\alpha}^{\infty} z^{-p-1} \int_{E_{\gamma, z}\cap \{\om: |f_{\gamma}(\om) - f (\om)|\le 2z\} } | \zeta_{\gamma}^{(z)}(\om) - \zeta^{(z)} (\omega)  |^{2p}d\omega dz 
\\ \le&~
\frac{1}{\alpha-1} \int_0^{\infty} z^{-p-1} \int_{\{\om: |f_{\gamma}(\om) - f (\om)|\le 2z\}} | f_{\gamma}(\om) - f(\omega)  |^{2p}d\omega dz 
\\ =&~
\frac{1}{\alpha-1} \int_{\Omega} |f_{\gamma}(\om) - f (\om)|^{2p} 
(    \int_{ \frac{| ( f_{\gamma} - f)(\omega) |}{2} }^{\infty} z^{-p-1} 
dz)
d\omega
\\=& ~
\frac{2^{p}}{p(\alpha-1)}\, \int_{\Omega} | ( f_{\gamma} - f)(\omega) |^p d\omega < \infty \,
\end{split}\]
for every $\gamma\in S$.
On the other hand,
\[\begin{split}
&\int_0^{\infty} z^{-p-1} \int_{E_{\gamma, z}\cap \{\om: |f_{\gamma}(\om) - f (\om)| > 2z\} } | (\zeta_{\gamma}^{(z)} - \zeta^{(z)} )(\omega)  |^{2p}d\omega dz 
\\ \le ~&\int_0^{\infty} z^{-p-1} \int_{\{\om: |f_{\gamma}(\om) - f (\om)| > 2z\} } (2z)^{2p} d\omega dz 
\\ = ~&\int_0^{\infty}  z^{-p-1} (2z)^{2p} \lambda_{|f_{\gamma} - f|  } (2z)dz  
\\ = ~& 2^{p+1} \int_0^{\infty} (2z)^{p-1} \lambda_{|f_{\gamma} - f|  } (2z)dz
\\ \overset{(*)}{=} \ &\frac{2^{p}}{p} \|f_{\gamma} -f \|_{p} < \infty 
\end{split}\]
for every $\gamma\in S$, where $(*)$ follows from \eqref{FirstLemma}.
Hence 
\begin{equation}\label{eq:p14}
\sum_{\gamma\in S} \int_{\alpha}^{\infty} z^{-p-1} \int_{E_{\gamma, z}} | (\zeta_{\gamma}^{(z)} - \zeta^{(z)} )(\omega)  |^{2p}\, d\om\, dz < \infty .
\end{equation}
Next we turn our attention to the complement sets $E_{\gamma, z}^c$. The estimates will be similar to the last case in the previous section.

One can see from the definitions that if $\alpha^k\le z < \alpha^{k+1}$ for some $k\in\mathbb{N}_1$, then $E_{\gamma, z}^c = \emptyset$ for any $\gamma\in S$. Thus, we assume $k\in\mathbb{N}_2$. If $k-1\in\mathbb{N}_1$, we have $E_{\gamma, z}^c\subset (A_k-A_{k+2})\cup s(A_k-A_{k+2})$, and so $\mu(E_{\gamma, z}^c)\le 2(1+\delta) \mu(A_{k+1})$; if $k-1\in\mathbb{N}_2$, then $E_{\gamma, z}^c\subset (A_{k-1}-A_{k+2}) \cup s(A_{k-1}-A_{k+2})$, and so $\mu(E_{\gamma, z}^c)\le 2(1+\delta)^2 \mu (A_{k+1})$, for every $\gamma\in S$. Thus, in either case, by \eqref{eq:ineq-N_2}, we get for every $\gamma\in S$,
\[
\mu(E_{\gamma, z}^c)\le \frac{16}{7\ep}(1+\delta)^2 \sup_{\gamma'\in S} \mu(\gamma' A_{k+1} - A_{k}) ,
\]
and therefore,
\[\begin{split}
\int_{E_{\gamma, z}^c} |(\zeta_{\gamma}^{(z)} - \zeta^{(z)})(\omega)|^{2p} d\omega 
\le 
(2z)^{2p}\frac{16}{7\ep}(1+\delta)^2 \sup_{\gamma'\in S} \mu(\gamma' A_{k+1} - A_{k}) .
\end{split}\]
Since $\gamma' A_{k+1}-A_{k}\subseteq \{\om : |(f_{\gamma'} - f) (\om)|>z(1-\frac{1}{\alpha})\}$ for all $\gamma'\in S$, 
using \eqref{FirstLemma} we get
\[\begin{split}
&~\sum_\gamma \, \int_{\alpha}^{\infty} z^{-p-1} \int_{E_{\gamma, z}^c} |(\zeta_{\gamma}^{(z)} - \zeta^{(z)} )(\omega)|^{2p}d\omega dz 
\\ {\le}  \, &~
\frac{16}{7\ep}(1+\delta)^2 \sum_\gamma \, \int_{\alpha}^{\infty} z^{-p-1} (2z)^{2p} \lambda_{|f_{\gamma} - f| } (z(1-\frac{1}{\alpha})) dz 
\\ = &~
\frac{16}{7\ep}(1+\delta)^2 2^{2p} (\frac{\alpha}{\alpha-1})^p\sum_\gamma \, \int_{\alpha}^{\infty} u^{p-1} \lambda_{|f_{\gamma} - f| } (u) du 
\\ = &~
\tilde c \sum_\gamma \, \|f_{\gamma}-f\|_{p} < \infty ,
\end{split}\]
where $\tilde c$ is a constant; this, together with \eqref{eq:p14} concludes the proof of the finiteness of the integral \textcircled{{\tiny II}}, and therefore completes the proof of the lemma.
\end{proof}

\bigskip

\noindent
{\it Proof of Theorem \ref{Main}.}\
Let $G$ be a topological group that has a finite Kazhdan set. Let $\Omega$ be a standard Borel space, $\mu$ a finite measure on $\Omega$, and $\pi: G \act L_p(\Omega,\mu)$ a linear isometric action whose BL action $G \act L_p(\Omega,\mu)$ is measure preserving. 
By ergodicity, if $\mu$ has atoms then $\Omega$ is a finite set, and the action $G \act L_p(\Omega,\mu)$ factors through a finite quotient of $G$, hence $H^1(G, \pi)=0$.

So, we may assume $(\Omega,\mu) = ([0,1], {\rm Lebesgue})$. By \cite[Corollary 2]{MarSal} the claim holds for any $p\in [1, 2]$. Thus, we assume $p>2$. Let $b: G\to L_p(\Omega,\mu)$ be a $\pi$-cocycle. Since $L_p(\Omega,\mu)\subset L_2(\Omega,\mu)$, we may consider $b$ as a cocycle for the corresponding action $\pi^2:G \act L_2(\Omega,\mu)$. Since $G$ has property (T), there is $f\in L_2(\Omega,\mu)$ such that $b_\gamma = f_\gamma - f$ for all $\gamma\in G$. Since $\big| |f_\gamma| - |f| \big| \leq |f_\gamma - f| =|b_\gamma| \in L_p(\Omega,\mu)$, Lemma \ref{lem:main} implies $|f| \in L_p(\Omega,\mu)$, which completes the proof. \qed

\end{document}